\documentclass[11pt]{amsart}

\usepackage{amsmath}
\usepackage{mathtools}
\usepackage{fullpage}
\usepackage{xspace}
\usepackage[psamsfonts]{amssymb}
\usepackage[latin1]{inputenc}
\usepackage{graphicx,color}
\usepackage{hyperref}
\usepackage{graphicx}
\usepackage{xcolor}

\usepackage{amsmath}%
\usepackage{amsthm}%
\usepackage{amscd}
\usepackage{amsfonts}%
\usepackage{amssymb}%
\usepackage{graphicx}
\usepackage{stmaryrd}
\usepackage{mathrsfs}
\usepackage{tikz}
\usetikzlibrary{matrix,arrows}
\usepackage{tikz-cd}
\usepackage[all]{xy}
\usepackage{bbm} 
\usepackage{longtable}

\usepackage{lineno} 

\newtheorem{theorem}{Theorem}[section]

\newtheorem{example}{Example}

\newtheorem{lemma}[theorem]{Lemma}

\newtheorem{proposition}[theorem]{Proposition}
\theoremstyle{remark}
\newtheorem{remark}[theorem]{Remark}

\numberwithin{equation}{section}

\newcommand{\C}{\mathbb{C}}

\newcommand{\N}{\mathbb{N}}

\newcommand{\Q}{\mathbb{Q}}
\newcommand{\R}{\mathbb{R}}

\newcommand{\Z}{\mathbb{Z}}

\newcommand{\Zp}{\mathbb{Z}_p}
\newcommand{\Li}{\mathrm{Li}}
\newcommand{\dimH}{\mathrm{dim}_{\mathrm{H}}}

\newcommand\rquot[2]{
  \mathchoice
  {
    \text{\raise0.5ex\hbox{$#1$}\big/\lower0.5ex\hbox{$#2$}}%
  }
  {
    #1\,/\,#2
  }
  {
    #1\,/\,#2
  }
  {
    #1\,/\,#2
  }
}

\newcommand\lrquot[3]{
  \mathchoice
  {
    \text{\lower0.5ex\hbox{$#1$}\big\backslash\raise0.5ex\hbox{$#2$\!}\big/
      \lower0.5ex\hbox{\!\!$#3$}}%
  }
  {
    #1\,\backslash\,#2\,/\,#3
  }
  {
    #1\,\backslash\,#2\,/\,#3
  }
  {
    #1\,\backslash\,#2\,/\,#3
  }
}

\newcommand\lquot[2]{
  \mathchoice
  {
    \text{\lower0.5ex\hbox{$#1$}\big\backslash\raise0.5ex\hbox{$#2$}}%
  }
  {
    #1\,\backslash\,#2
  }
  {
    #1\,\backslash\,#2
  }
  {
    #1\,\backslash\,#2
  }
}

\usepackage{fontawesome}

  \DeclareFontFamily{U}{wncy}{}
    \DeclareFontShape{U}{wncy}{m}{n}{<->wncyr10}{}
    \DeclareSymbolFont{mcy}{U}{wncy}{m}{n}
    \DeclareMathSymbol{\Sha}{\mathord}{mcy}{"58}
\begin{document}

\title{Multifractal analysis of power means for the Schneider map on $p\mathbb{Z}_{p}$}

\date{\today}



\author{ Matias Alvarado and Nicol\'as Ar\'evalo-Hurtado}	
\address{Instituto de Matem\'aticas, Universidad de Talca, Talca, Chile.}
\email[M. Alvarado]{matias.alvarado@utalca.cl}
\address{Universidad Escuela Colombiana de Ingenier\'ia Julio Garavito, Bogot\'a, Colombia.}

\email[N. Ar\'evalo]{nicolas.arevalo-h@escuelaing.edu.co}
\urladdr{\url{https://sites.google.com/view/nicolasarevalomath/}}

\begin{abstract}
We study the asymptotic power means of the coefficients associated with the Schneider continued fraction map on $p\Z_p$. Using tools from thermodynamic formalism, we compute the Hausdorff dimension of the corresponding level sets and obtain explicit formulas for the associated multifractal spectra. The locally constant nature of the geometric potential enables a precise description in terms of polylogarithm functions, in sharp contrast with the classical real setting.

\end{abstract}

\maketitle
\vspace{-10mm}
\section{Introduction}\label{introduction}
The study of the digits arising in continued fraction expansions has long been a central topic in number theory and dynamical systems. Beyond their intrinsic arithmetic significance, these coefficients encode subtle metric and statistical properties of numbers. A classical result of Khinchin \cite{Khinchin35} (see also \cite{Khinchin}) establishes that, for Lebesgue-almost every $x\in (0,1)$, the asymptotic geometric mean of the digits appearing in the continued fraction of $x$ converges to a universal constant. This result initiated a broad line of research devoted to the statistical behavior of continued fraction digits, including the study of various types of asymptotic averages from an ergodic-theoretic perspective \cite{HMR, RyllNardzewski}.

While almost-everywhere results provide valuable information, they fail to capture the full complexity of continued fraction expansions. A more refined approach consists of studying the level sets of points for which a given asymptotic average attains a prescribed value. Although such sets typically have zero Lebesgue measure, they exhibit a rich fractal structure. Their size is naturally described in terms of Hausdorff dimension, leading to a multifractal analysis of the underlying dynamical system. This theory has been extensively developed for classical continued fractions, particularly in connection with the Gauss map (see, for instance, \cite[Section 6]{iommi2015multifractal}).

Among the various quantities that can be studied, the asymptotic power means of the digits provide a natural and flexible family, interpolating between classical notions such as the arithmetic, geometric, and harmonic means. However, the multifractal analysis of these power means presents substantial technical challenges. In the classical setting, standard methods (drawing on the work of Pesin and Weiss \cite{yh} and on the generalization by Jaerisch and Kesseb{\"o}hmer in \cite{jaerisch2011regularity} for the Gauss map) rely on a detailed understanding of topological pressure and Lyapunov exponents associated with appropriate invariant measures. These quantities are typically defined only implicitly and are notoriously difficult to compute explicitly, which severely limits the possibility of obtaining closed-form expressions for the corresponding multifractal spectra.

In this paper, we show that these difficulties can be overcome in the $p$-adic setting by considering the continued fraction algorithm introduced by Schneider. A key feature of this dynamical system is that its associated geometric potential is locally constant, which allows for an explicit computation of the topological pressure and a precise control of the Lyapunov spectrum. As a consequence, the multifractal analysis of power means becomes tractable and leads to explicit formulas, in contrast to the classical real case.

We now briefly describe the setting. Let $p$ be a prime number and consider the Schneider map $T_{p}:p\Z_{p}\rightarrow p\Z_{p}$ defined by $T_{p}(0)=0$ and, for $x\neq 0$
  \begin{align*}
      T_{p}(x)=\dfrac{p^{a_{1}(x)}}{x}-b_{1}(x),
  \end{align*} where $a_{1}(x)=v_{p}(x)$ denotes the $p$-adic valuation and $b_1(x)\in\{1,2,...,p-1\}$ is uniquely determined by the congruence $b_{1}(x)\equiv p^{a_{1}(x)}/x$ (mod $p$). More generally, we define $a_i(x)=a_1(T_p^{i-1}x)$ and $b_i(x)=b_1(T_p^{i-1}x).$ These coefficients allow us to construct the continued fraction expansion of $x$, as shown in equation \eqref{SchneiderContinued}.

For each $q\in \R\setminus\{0\}$ and for points $x\in p\Zp$ whose forward orbit never vanishes, we define the asymptotic $q$-power mean by   
\[M_{q}(x)=\lim_{n\to \infty} \left(\frac{1}{n}\sum^{n}_{i=1}a_{i}(x)^{q}\right)^{\frac{1}{q}},\]
whenever this limit exists.

In the case $q=0$, we instead consider the asymptotic geometric mean
\[M_0(x)=\lim_{n\to \infty}\prod_{i=1}^{n}a_{i}(x)^{\frac{1}{n}}.\]

Given a real parameter $q$ and a value $\beta\geq 1$, our goal is to study the Hausdorff dimension of

\[K_q(\beta)=\left\{x\in p\Z_p: M_{q}(x)=\beta \right\}.\]

The case of the arithmetic mean ($q=1$) has been previously investigated by Hu, Yu, and Zhao in \cite{hyz}, where an explicit formula for the Hausdorff dimension of $K_{1}(\beta)$ was obtained. However, extending these results to general power means is far from straightforward and requires new ideas.

Our approach is based on a thermodynamic formalism framework. More precisely, we show that the equilibrium state associated with the potential defining the $q$-power mean can be identified with the equilibrium state of a suitable geometric potential. By means of Birkhoff's ergodic theorem, this allows us to relate each $q$-power mean to an appropriate arithmetic mean, thereby reducing the problem to the study of the Lyapunov spectrum. This 
reduction, ctogether with recent results obtained in \cite[Theorem 1.1]{paper1}, leads to an explicit and computable formula for the Hausdorff dimension of the level sets $K_q(\beta)$.

We denote by $\dim_\mathrm{H}(\cdot)$ the Hausdorff dimension. See \cite[Section 2]{paper1} for a brief discussion on Hausdorff dimension with respect to the $p$-adic norm.

The main result of this article is the following.

\begin{theorem}\label{MainTheoremPaper2}
    Let $q\in \mathbb{R}$ and $\beta\geq 1$. There exists a unique $\alpha_{\beta}\geq \log p$ such that 
    \begin{align}\label{EcuTheoremPrincipal 1}
        \dimH\left(K_{q}(\beta)\right)=\dfrac{\alpha_{\beta}\log\alpha_{\beta}-(\alpha_{\beta}-1)\log(\alpha_{\beta}-1)+\log(p-1)}{\alpha_{\beta} \log p},
    \end{align}where $\alpha_{\beta}$ is such that 
    \begin{align}\label{EcuTheoremPrincipal 2}
        \beta^{q}=
            \dfrac{\log p}{\alpha_\beta-\log p}\Li_{-q}\left(\dfrac{\alpha_\beta-\log p}{\alpha_\beta}\right)
    \end{align}for $q\neq 0$, and
    \begin{align}\label{EcuTheoremPrincipal 21}
        \log(\beta)=-\dfrac{\log p}{\alpha_{\beta}-\log p}\left.\dfrac{d}{ds}\Li_{s}\left(\dfrac{\alpha_\beta-\log p}{\alpha_\beta}\right)\right|_{s=0}
    \end{align} for $q=0$, where $\Li$ denotes the polylogarithm function.
\end{theorem}

This result provides a complete and explicit description of the multifractal spectrum associated with power means in the $p$-adic setting, highlighting a sharp contrast with the classical real case, where such formulas remain out of reach. It is worth to say that it is possible to study $\dimH\left(K_{q}(\beta)\right)$ by means of classical multifractal analysis techniques (see \cite{jaerisch2011regularity,yh}). Nevertheless, this approach does not yield expressions for the Hausdorff dimension in terms of explicit functions such as the polylogarithm.


\section{Preliminaries}

This section collects the main preliminaries required for our study. We begin by introducing the Schneider map and the associated power means. This is followed by a formal statement of the problem under consideration. We then review its Lyapunov spectrum and discuss its relation to power means. We conclude with a brief overview of polylogarithmic functions.

\subsection{$p$-adic numbers and Schneider map}\label{padic numbers and schneider maps}

Let $p$ be a prime number and let $v_p$ denote the $p$-adic valuation on $\Q$, defined by the property that for each $x\in \Q$ there exists a unique integer
$v_p(x)$ such that $x=p^{v_p(x)}\frac{m}{n}$ with $(p,mn)=1$. The $p$-adic absolute value on $\Q$ is defined by $|0|_p=0$ and $|x|_p=p^{-v_p(x)}$ if $x\neq 0$. The field $\Q_p$ of $p$-adic numbers is the completion of $\Q$ with respect to this norm. The set of all elements $x\in \Q_p$ such that $|x|_p\leq 1$ is denoted by $\Z_p$, and the subset $p\Z_p$ consists of those with $|x|_p<1$. Explicitly, any element in $\Q_p$ can be expressed as $\sum_{n\geq n_0}c_np^{n}$, where $n_0$ is an integer depending only on $x$ and $c_n\in \{0,1,...,p-1\}$. An element $x$ in $\Z_p$ can be written as $\sum_{n\geq0} c_np^n$. Thus, $p\Zp$ consists of all elements in $\Z_p$ with $c_0=0.$ 
Since $\Q_p$ is locally compact, there exists a Haar measure $\mu_p$ normalized by $\mu_{p}(p\Z_p)=1$. A detailed exposition of Haar measures can be found in \cite{folland}.

Recall from Section \ref{introduction} that for $x\in p\Z_p\setminus\{0\}$, the Schneider map $T_{p}:p\Z_{p}\rightarrow p\Z_{p}$ is defined by
  \begin{align*}
      T_{p}(x)=\dfrac{p^{a_{1}(x)}}{x}-b_{1}(x),
  \end{align*} where $a_{1}(x)=v_{p}(x)$ and $b_1(x)\in\{1,2,...,p-1\}$ is determined by $b_{1}(x)\equiv p^{a_{1}(x)}/x$ (mod $p$). The coefficients $a_i(x)$ and $b_i(x)$ are defined by $a_1(T_p^{i-1}x)$ and $b_i(x)=b_1(T_p^{i-1}x)$ respectively. For every $x\in p\Z_p\setminus \bigcup_{k\in \N}T_p^{-k}(0)$, and for each $n\in \N$, one can write $x$ as a continued fraction as follows:
  
  \begin{equation}\label{SchneiderContinued}
      x= \cfrac{p^{a_{1}(x)}}{b_{1}(x)+\cfrac{p^{a_{2}(x)}}{b_{2}(x)+\cfrac{p^{a_{3}(x)}}{\ddots + \cfrac{p^{a_{n}(x)}}{b_{n}(x)+T_{p}^n(x)}}}},
  \end{equation}
  
  We denote by $F=\bigcup_k T_p^{-k}(0)$ the set of elements in $p\Z_p$ whose continued fraction expansion is finite. Since we are interested in the asymptotic behavior of the coefficients, we restrict our attention to elements in $p\Z_p\setminus F.$

\subsection{Power means} 
Let $q\in \mathbb{R}\setminus\{ 0\}$, $n\in \N$,  and let $a_{1},a_{2},...,a_{n}$ be a finite sequence of positive numbers. The \textit{$q$-power mean} of $\{a_{i}\}^{n}_{i=1}$ is defined as
\begin{align*}   
M_{q}\left(\{a_1,...,a_n\}\right)=\left(\dfrac{1}{n}\sum^{n}_{i=1}a_{i}^{q}\right)^{\frac{1}{q}}.
\end{align*} 

\begin{remark}$M_1(\cdot)$ and $M_{-1}(\cdot)$ coincide with the arithmetic and harmonic means, respectively. Moreover, an application of L'H\^opital's rule shows that $\lim_{q\to 0}M_q(\cdot)$ coincides with the geometric mean.
\end{remark}
Note that, for fixed $q\in \R\setminus \{0\}$, the quantity $M_{q}\left(\left\{a_1,...,a_n\right\}\right)^{q}$ is simply the arithmetic mean of the sequence
$\{a_{1}^{q},...,a_{n}^{q}\}$. Define the geometric potential $\psi\colon p\Z_p\to \R$ by $\psi(x)=p^{a_1(x)}$. Then, for every $n\in \N$ and $x\in p\Zp\setminus F$  
\begin{align}\label{QpowermeanP}
    M_{q}\left(\left\{a_1(x),...,a_n(x)\right\}\right)^{q}=\dfrac{1}{n(\log p)^{q}}\sum^{n-1}_{k=0}\left(\log \psi (T^{k}x)\right)^{q}.
\end{align}
Taking the logarithm in the case of the geometric mean (instead of raising to the $q$-th power), we obtain
\begin{align}\label{QpowermeanG}
    \log M_0(\{a_1(x),...,a_n(x)\})=-\log\log p+\dfrac{1}{n}\sum_{k=0}^{n-1} \log \log \psi (T^{k}x).
\end{align}
    
    Equivalently, the potential associated with the $q$-power mean is $(\log\psi)^{q}$ for $q\neq 0$, and $\log \log \psi$ for $q=0$.
Recall from the introduction that the \textit{asymptotic $q$-power mean} and \textit{asymptotic geometric mean} of $x$ are defined, respectively, by
\begin{align*}
    M_{q}(x)=\lim_{n\to \infty} \left(\dfrac{1}{n}\sum_{k=1}^n a_k(x)^q \right)^{\frac{1}{q}}\text{, and }M_{0}(x)=\lim_{n\to\infty}\prod^{n}_{k=1}a_{k}(x)^{\frac{1}{n}}.
\end{align*}

In Proposition \ref{PropCalculosHar} (see section \ref{sectionproof}) we show that $M_q(x)$ is constant for Haar-almost every point $x$. This constant can be described using the polylogarithm function.

\subsection{Lyapunov spectrum and thermodynamic formalism}

We begin by introducing a family of subspaces of $p\Z_p$ with bounded coefficients, which will play a central role in the proof of the main theorem. For each $n\in \N$,  define  
$$p\Z_{p,n}=\{x \in p\Z_p\setminus F: a_i(x) \leq n \text{ for all }i\in \N\}.$$

Let $\rho:p\Zp \setminus F\to \mathbb{R}$ be a continuous function. The \emph{topological pressure of} $\rho$ is defined by  

\begin{align}\label{DefPressure2}
    P(\rho) := \sup \left\{ h_{\mu} + \int \rho \, d\mu : \mu \in \mathcal{M}(p\Zp\setminus F,T_{p}), -\int \rho \, d\mu < \infty \right\},
\end{align}where \( \mathcal{M}(p\Zp\setminus F,T_{p}) \) denotes the set of \( T_{p} \)-invariant probability measures on $(p\Zp\setminus F,T_{p})$, and \( h_{\mu} \) is the measure-theoretic entropy of \( \mu \). Any measure attaining the supremum is called an \textit{equilibrium state} for \( \rho\). For a potential $\varphi$ defined on $p\Z_{p,n}$, we denote by $P_n(\varphi)$ the corresponding topological pressure on $p\Z_{p,n}$. 

Since the geometric potential $\psi$ is locally constant, the pressure $P([\log\psi]^{q})$ can be computed via periodic points (see \cite[Cor. 1]{bs}). In particular, we have
\begin{align}\label{EcuPressure2}
    P([\log\psi]^{q})=\lim_{n\rightarrow \infty}\dfrac{1}{n}\log\left(\sum_{T_{p}^{n}x=x}e^{-S_{n} [\log \psi(x)]^{q}}\right),
\end{align} where $S_{n} [\log \psi(x)]^{q}=\sum^{n-1}_{k=0}[\log\psi\circ T^{k}_{p}x]^{q}$ denotes the $n$-th Birkhoff sum of $[\log \psi]^{q}$ with respect to $T_p$.

An analogous expression holds for $P(\log\log \psi)$ by replacing $[\log \psi]^{q}$ for $\log\log \psi$ in equation \eqref{EcuPressure2}. 
For every $x\in p\Z_p \setminus F$, the Lyapunov exponent is defined by

\begin{align*}
     \lambda_{p}(x)=\lim_{n\rightarrow{\infty}}\dfrac{1}{n}S_{n}\log\psi (x),
\end{align*} 
whenever this limit exists.  By the definition of $\psi$, we have 
\begin{align}
    \lambda_{p}(x)=\lim_{n\rightarrow{\infty}}\log p \cdot \dfrac{a_{1}(x)+a_{2}(x)+\cdots +a_{n}(x)}{n}.
\end{align} 
In particular, this shows that $\lambda_{p}(x)=M_{1}(x)\log p$. 

In the multifractal analysis of the Lyapunov exponent, one studies, for each $\alpha\in \R$, the level sets \[K_{1}\left(\frac{\alpha}{\log p}\right)=\left\{x\in p\Z_{p}\setminus F: M_{1}(x)=\frac{\alpha}{\log p}\right\}.\] The \textit{Lyapunov spectrum} is defined by \[L_{p}(\alpha)=\mathrm{dim}_{\mathrm{H}}\,K_{1}\left(\frac{\alpha}{\log p}\right).\] 
The main result in \cite{paper1} establishes the following 
\begin{theorem}[\cite{paper1}, Theorem 1.1 and 1.2]\label{Theorem Paper 1}
      The Lyapunov spectrum $L_{p}$ is real analytic on $[\log(p),\infty)$. For each $\alpha\geq  \log p$
      \begin{align}\label{EcuSInCalculoExplicito}
        L_{p}(\alpha)&=\dfrac{\log(p-1)+\log(\alpha-\log p)-\log\log p+\alpha\log_{p} \alpha-\alpha\log_{p}(\alpha-\log p)}{\alpha}\\&=\dfrac{1}{\alpha}\inf\left\{P(-s\log \psi)+s\alpha:s>0\right\},
    \end{align} The infimum is attained at a unique $s_{\alpha}>0$ such that $\frac{d}{dt}P(-s\log \psi)|_{s=s_{\alpha}}=-\alpha$. Moreover, there exists a unique equilibrium state $\mu_{s_{\alpha}}$ for $-s_{\alpha}\log \psi$ such that $\mu_{s_{\alpha}}(K_{1}(\alpha/\log p))=1$.
\end{theorem}
An immediate consequence of Theorem \ref{Theorem Paper 1} is the following.
\begin{lemma}\label{Corollary 1}
    Let $\rho:p\Z_{p}\setminus F\to\mathbb{R}$ be a continuous function. Let $t\in \R$ be such that $t\rho$ has a unique equilibrium state $\mu_{t}$. Let $\beta=\mu_{t}(\rho)$ and $\alpha=\mu_{t}(\log \psi)$. If $\alpha<\infty$, then the Haussdorff dimension of the set of points
    \begin{align*}
        M(\beta):=\left\{x\in p\Z_{p}:\lim_{n\to\infty}\dfrac{S_{n}\rho(x)}{n}=\beta\right\},
    \end{align*}is the same as $K_{1}(\alpha/\log p)$, that is
    \begin{align*}
        \dimH(M(\beta))=\dfrac{\alpha\log\alpha-(\alpha-1)\log(\alpha-1)+\log(p-1)}{\alpha \log p}.
    \end{align*}
\end{lemma}
\begin{proof}
    Since $\mu_{t}$ is ergodic, then $\log\psi$ is constant almost every point. Therefore, $\mu_{t}$ must be the only equilibrium state of $-s_{\alpha}\log\psi$. Thus $M(\beta)$ and $K_{1}(\alpha)$ are the support of $\mu_{t}$. The result follows by Theorem \ref{Theorem Paper 1}.
\end{proof}

\subsection{Polylogarithm function}
The polylogarithm function is a classical special function that arises naturally in number theory, algebraic geometry, and mathematical physics. For $s\in \C$ and $z\in \C$ with 
$|z|<1$, it is defined by the absolutely convergent power series
\[\Li_s(z)=\sum_{n=1}^\infty \dfrac{z^n}{n^s}.\]

The polylogarithm interpolates several classical
functions. For instance, when $s=1$, one recovers
the logarithmic function $\Li_1(z)=-\log(1-z)$,
while for positive integers $s\geq 2$, the functions $\Li_s(z)$ are closely related to multiple zeta values and appear in the study of special values of 
$L$-functions. In particular, at $z=1$, one has
$\Li_s(1)=\zeta(s)$, for $\Re(s)>1$, where $\zeta(s)$ denotes the Riemann zeta function.

Let $m\geq 0$ be an integer, and let $0\leq k \leq m$. Denote by $A(m,k)$ the Eulerian number, that is, the number of permutations of $\{1,...,m\}$ with exactly $k$ ascents. Then  $\Li_{-m}(z)$ admits the representation
\begin{equation}\label{Li in negative numbers}
\Li_{-m}(z)=\dfrac{z}{(1-z)^{m+1}}\sum_{k=0}^m A(m,k)z^k.  \end{equation}

Moreover,  certain special values of its derivative are related to classical series. A direct computation from the definition yields the following identity.

\begin{proposition}\label{derivative in zero}
The following identity holds:    
\[\left.\dfrac{d}{ds}\Li_s(1/p)\right|_{s=0}=-\sum_{k=1}^\infty \dfrac{\log k}{p^k}.\]
\end{proposition}



The polylogarithm satisfies the following identity (see \cite[equation 25.12.12.]{handbookNIST}) 

\begin{equation}\label{equation Li Gamma}\Li_s(z)=\Gamma(1-s)(-\log z)^{s-1}+\sum_{k=0}^\infty \zeta(s-k)\dfrac{(\log z)^k}{k!},
\end{equation}
where $\Gamma$ denotes the Gamma function. This representation is particularly well suited for studying the asymptotic behavior as $z \to 1$.

\begin{proposition}\label{PropositionLimitLi}
    Let $s\in \R$, the following limit holds.
    \[\lim_{z\to \infty }\dfrac{1}{z}\Li_s\left(\dfrac{z-\log p}{z} \right)=\begin{cases}
        0 &\text{ if }s>0\\
        \dfrac{1}{\log p} &\text{ if }s=0\\
        +\infty &\text{ if }s<0.
    \end{cases}\]
   
\end{proposition}

\begin{proof}
    First we suppose $s>1$. Since $\Li_s\left(\dfrac{z-\log p}{z}\right)\to \zeta(s)$ as $z\to \infty$, it follows that 
    $$\lim_{z\to \infty}\dfrac{1}{z}\Li_s\left(\dfrac{z-\log p}{z} \right)=0.$$
    The case $s=1$ follows from the identity $\Li_s(z)=-\log (1-z)$ by a direct computation.

    Now assume $s<1.$ Note that $\Gamma(1-s)$ has no poles in this domain. Since $\zeta(s)\to 0$ as $s\to -\infty$ (from the functional equation, see \cite[equation 3, p.10]{edwardzetafunction}. By equation \eqref{equation Li Gamma}, it follows that \[\Li_s\left(1-\dfrac{\log p}{z} \right)= \left(-\log \left(1-\dfrac{\log p}{z} \right)\right)^{s-1}\Gamma(1-s)+O(1).\]
Consider the variable $t=1-(\log p)/z$, then $t\to 1^-$ as $z\to \infty$. Then, 
\[\lim_{z\to \infty}\dfrac{1}{z}\Li_s\left(1-\dfrac{\log p}{z}\right)=\lim_{t\to 1^{-}}\dfrac{(1-t)\Gamma(1-s)}{(\log p)  (-\log t)^{1-s}}.\]
    Using the Taylor expansion of $-\log t$ we conclude that this limit is equal to $1/\log p$ if $s=0$ and $+\infty$ is $s<0$.
    This conclude the proof.
\end{proof}

\begin{lemma}\label{lemma derivada en s} The following identity holds.
      \[\dfrac{1}{z}\left.\dfrac{d}{ds}\Li_s\left(\dfrac{z-\log p}{z} \right)\right|_{s=0}=\dfrac{1}{\log p}\left( \gamma +\log \log p-\log z \right)+O\left( \dfrac{\log z}{z}\right),\]
      where $\gamma$ denotes the Euler-Mascheroni constant. In particular

\[\lim_{z\to \infty}\dfrac{1}{z}\left.\dfrac{d}{ds}\Li_s\left(\dfrac{z-\log p}{z} \right)\right|_{s=0}=-\infty.\]
\end{lemma}

\begin{proof}
Recall that $\Gamma(1-s)=1+\gamma s+O(s^2)$. Let $w=1-(\log p) /z$. Differentiating $\Li_s(w)$ from equation \eqref{equation Li Gamma}, we obtain
\[\dfrac{d}{ds}\Li_s(w)=(\gamma +O(s))(-\log w)^{s-1}+(1+\gamma s+O(s^2))\cdot \log (-\log w)\cdot (-\log w)^{s-1}+O(1).\]
Evaluating at $s=0$, we obtain
\[\left. \dfrac{d}{ds}\Li_s(w)\right|_{s=0}=\dfrac{\gamma+\log(-\log w)}{-\log w}+O(1).\]

Recalling that $w=1-(\log p)/z$, and $-\log w=1-w+O((1-w)^2)$ as $w\to 1^{-}$, we conclude 
\[\dfrac{d}{ds}\Li_s\left(1-\dfrac{\log p}{z} \right)=\dfrac{z(\gamma+\log \log p-\log z)}{\log p}+O(\log z),\]
which yields the desired estimate after dividing by 
$z$.

\end{proof}

\section{Power means and proof of main Theorem}\label{sectionproof}

We begin this section computing for Haar-almost every $x\in p\Z_p\setminus F$, the quantity $M_q(x)$.

\begin{proposition}\label{PropCalculosHar}
For $\mu_p$-almost every $x\in p\Zp\setminus F$, the following equations hold:
    \begin{equation}\label{Kqctp}
    M_{q}(x)=\left((p-1)\Li_{-q}(1/p)\right)^{\frac{1}{q}} \text{  for $q\neq 0$, and }
    \end{equation} 
    \begin{equation}\label{K0ctp}
    M_{0}(x)=\exp\left(-(p-1)\left.\dfrac{d}{ds}\Li_s(1/p)\right|_{s=0}\right)
    \end{equation}
\end{proposition}

\begin{proof}
    From equation \eqref{QpowermeanP} we have
    \[M_q(x)^q=\dfrac{1}{(\log p)^q}\lim_{n\to \infty}\dfrac{1}{n}\sum_{k=0}^{n-1}(\log \psi (T^ix))^q. \]
    By Birkhoff ergodic theorem, this last expression coincides for Haar-almost every $x$ with
    \begin{align}\label{IntKq}
        \dfrac{1}{(\log p)^q}\int_{p\Z_p} (\log \psi(x))^q d\mu_p.
    \end{align}
    On the other hand, this integral can be computed as 
    \begin{align*}
        \dfrac{1}{(\log p)^q}\int_{p\Z_p} (a_k\log p)^q d\mu_p&=\sum_{k\geq 1}k^q \mu_p \left(p^k\Z_p \setminus p^{k+1}\Z_p \right)\\
        &=(p-1)\sum_{k\geq 1}\dfrac{k^q}{p^k}\\
        &=(p-1)\Li_{-q}(1/p).
    \end{align*}
Equation \eqref{Kqctp} follows from equation \eqref{IntKq}. The proof of equation \eqref{K0ctp} follows the same lines as the previous one.

From equation \eqref{QpowermeanG}, and Birkhoff ergodic theorem we have for $\mu_p$-almost every $x\in p\Z_p\setminus F$
 \[\log M_0(x)=-\log \log p +\int_{p\Z_p} \log \log \psi(x)d\mu_p\]
 evaluating the integral, we obtain
 \[\log M_0(x)=(p-1)\sum_{k=1}^\infty \dfrac{\log k}{p^k}.\]
 We conclude the proof by proposition \ref{derivative in zero}.
\end{proof}

Now we prove the main theorem.

\begin{proof}[Proof of Theorem \ref{MainTheoremPaper2}]
    Fix $\alpha\geq \log p$. We begin by deriving equations \eqref{EcuTheoremPrincipal 2} and \eqref{EcuTheoremPrincipal 21}. By Theorem \ref{Theorem Paper 1}, there exists a unique $s_{\alpha}>0$ and a unique equilibrium state $\mu_{s_{\alpha}}$ for the potential $-s_{\alpha}\log\psi$ such that $\alpha=\int \log\psi d\mu_{s_{\alpha}}$. Our goal is to compute the quantity $\int [\log\psi]^{q}d\mu_{s_{\alpha}}$, which determines the $q$-power mean associated with $\mu_{s_{\alpha}}$. To this end, we study the derivative of the topological pressure restricted to the subspaces $p\Z_{p,n}$. By \cite[Theorem 3.3]{paper1}, there exists $N\in \N$ such that $p\Z_{p,n}\cap K_{1}(\frac{\alpha}{\log p})\neq \emptyset$ for all $n\geq N$. For each $n\geq N$, let $s_{\alpha,n}>0$ denote the unique solution of $\frac{d}{ds}P_{n}(-s\log\psi)\big\lvert_{s=s_{\alpha,n}}=-\alpha$, and let $\mu_{\alpha,n}$ be the corresponding equilibrium state of $-s_{\alpha,n}\log \psi$. Then $\mu_{\alpha,n}(\log\psi)=\alpha$. Moreover, by \cite[Theorem 5.6.5]{pu}, we have
    \begin{align}\label{EcuProof1}
        \int[\log\psi]^{q}d\mu_{\alpha,n}=\left.\dfrac{d}{dt}P_{n}(-s_{\alpha,n}\log\psi+t[\log\psi]^{q})\right|_{t=0}.
    \end{align}

    Since the sequence $\{\mu_{\alpha,n}\}_{n\in \N}$ is  tight (see \cite[page 10]{paper1}), there exists a subsequence $\{\mu_{\alpha,n_k}\}_{k\in \N}$ converging in the weak-* topology to $\mu_{s_\alpha}$. Passing to the limit in equation \eqref{EcuProof1}, we obtain 
    \begin{align*}
        \mu_{s_{\alpha}}([\log\psi]^{q})&=\lim_{k\to \infty}\mu_{\alpha,n_{k}}([\log\psi]^{q})\\
        & = \lim_{k\to \infty}\dfrac{d}{dt}P_{n_{k}}(-s\log\psi+t[\log\psi]^{q})\big\lvert_{t=0}.
    \end{align*}
    Using the formula \eqref{EcuPressure2} together with the fact that $\psi$ is locally constant, we obtain
    \begin{align*}
        \mu_{s_{\alpha}}([\log\psi]^{q})&=\lim_{k\to \infty}\dfrac{d}{dt}\log\left(\sum^{n_{k}}_{a=1}p^{-s_{\alpha,n_{k}}a+ta^{q}\log(p)^{q}}\right)\bigg\lvert_{t=0}\\
        &= \lim_{k\to \infty}\dfrac{\sum^{n_{k}}_{a=1}p^{-s_{\alpha,n_{k}}a}a^{q}\log(p)^{q}}{\sum^{n_{k}}_{a=1}p^{-s_{\alpha,n_{k}}a}}.
    \end{align*}
    Since $\{s_{\alpha,n_{k}}\}_{k\in \N}$ converges to $s_{\alpha}$ (\cite[p.10]{paper1}) and the sums converge absolutely, we may pass to the limit and obtain
    
    \begin{align*}
        \mu_{s_{\alpha}}([\log\psi]^{q})= \dfrac{\sum^{\infty}_{a=1}p^{-sa}a^{q}\log(p)^{q}}{\sum^{\infty}_{a=1}p^{-sa}}.
    \end{align*}Using the explicit expression for $s_{\alpha}$ in terms of $\alpha$ ( see proof Theorem \ref{Theorem Paper 1} in \cite{paper1}), by replacing $s=\log_{p}(\frac{\alpha}{\alpha-\log(p)})$ we obtain
    \begin{align}\label{Ecuproof2}
        \mu_{s_{\alpha}}([\log\psi]^{q})=\dfrac{(\log p)^{q+1}}{\alpha-\log p}\sum^{\infty}_{a=1}a^{q}\left(\dfrac{\alpha-\log p}{\alpha}\right)^{a}=\dfrac{(\log p)^{q+1}}{\alpha-\log p}\Li_{-q}\left(\dfrac{\alpha-\log p}{\alpha}\right).
    \end{align} Given that $\frac{\alpha-\log p}{\alpha}< 1$, we get that $\mu_{s_{\alpha}}([\log\psi]^{q})<\infty$. 
    
     Now we show that all possible values of the $q$-power means are attained by varying the parameter $\alpha$ in equation \eqref{Ecuproof2}. Recall from equation \eqref{QpowermeanP} that the $q$-power mean of $\mu_{\alpha}$-almost every point is given by $\frac{1}{\log p}(\mu_{\alpha}([\log\psi]^{q}))^{\frac{1}{q}}$. On the one hand, by applying L'H\^opital's rule we obtain $\lim_{\alpha\to \log p^{+}}\mu_{\alpha}([\log\psi]^{q})=[\log p]^{q}$. On the other hand, $\lim_{\alpha\to \infty}(\mu_{\alpha}([\log\psi]^{q}))^{\frac{1}{q}}=\infty$ by Proposition \ref{PropositionLimitLi} for both positive and negative values of $q$. Therefore, for every admissible value $\beta$ of the asymptotic $q$-power mean, the ergodicity of the measures $\mu_{s_{\alpha}}$ implies the existence of a unique $\alpha_{\beta} \geq \log p$ such that $(\mu_{\alpha_{\beta}}([\log\psi]^{q}))^{\frac{1}{q}}=\beta\log p$. Moreover, by equation \eqref{Ecuproof2}, the parameter $\alpha_{\beta}$ satisfies the first case of equation \eqref{EcuTheoremPrincipal 2}.
     Furthermore, by Corollary \ref{Corollary 1}, we deduce that $K_{q}(\beta)$ coincide with $K_1(\frac{\alpha_{\beta}}{\log p})$, and hence satisfies equation \eqref{EcuTheoremPrincipal 1}.
        
     The argument for the geometric mean (i.e. $q=0$) is analogous. In this case, one replaces $(\log \psi)^q$ by $\log \log \psi$ and uses Lemma \ref{lemma derivada en s} to control the asymptotic behavior.  This yields equation \eqref{EcuTheoremPrincipal 21} and completes the proof.
\end{proof}

\section{Some explicit computations}
Here we apply Theorem \ref{MainTheoremPaper2} for several explicit values of $q$. In particular, using formula \eqref{Li in negative numbers} we can compute $\dim_\mathrm{H}K_{q}(\beta)$ for certain integer values of $q$.

\begin{example}
    We begin with the basic case $q=1$. It is worth noting that equation \eqref{EcuTheoremPrincipal 2} also holds for the arithmetic mean. Indeed, for the asymptotic arithmetic mean we obtain
    \begin{align*}
        \beta=\dfrac{(\log p)^{2}}{\alpha_{\beta}-\log p}\Li_{-1}\left(\dfrac{\alpha_{\beta}-\log p}{\alpha_{\beta}}\right)=\dfrac{(\log p)^{2}}{\alpha_{\beta}-\log p}\left(\dfrac{\dfrac{\alpha_{\beta}-\log p}{\alpha_{\beta}}}{\left(1-\dfrac{\alpha_{\beta}-\log p}{\alpha_{\beta}}\right)^{2}}\right)=\alpha_{\beta}.
    \end{align*}
\end{example}

\begin{example}
    For the Harmonic mean $(q=-1)$, equation \eqref{EcuTheoremPrincipal 2} yields that, for $\beta\geq1$, the associated parameter $\alpha_{\beta}$ satisfies

\[\beta^{-1}=\dfrac{\log p}{\alpha_\beta-\log p}\log \left(\dfrac{\alpha_\beta}{\log p} \right)\]

We can express $\alpha_\beta$ explicitly in terms of  $\beta$ using the Lambert $W$-function, defined as the inverse of $x\mapsto xe^x$ for $x>0$. A straightforward computation gives

    \begin{align*}
        \alpha_\beta=-\dfrac{\log p}{\beta \, W\left( -\dfrac{1}{\beta}e^{-1/\beta}\right)}.
    \end{align*}
\end{example}
\begin{example}
    For the Geometric mean, by equation \eqref{equation Li Gamma} together with the identities $\Gamma(1)=1$, $\Gamma'(1)=-\gamma$, Theorem \ref{MainTheoremPaper2} yields
    \[ \log \beta=\dfrac{-\gamma}{\log \left(1-\frac{\log p}{\alpha_\beta} \right)}+\log \left( -\log \left(1-\frac{\log p}{\alpha_\beta} \right)\right)+\sum_{k=0}^{\infty} \zeta'(-k)\dfrac{\log\left(1-\frac{\log p}{\alpha_\beta} \right)^k}{k!}.\]
\end{example}
   
\begin{example}
    For the quadratic mean ($q=2$) we compute $\alpha_{\beta}$ as follows
    \begin{align*}
        \beta^{2}=\dfrac{\log p}{\alpha_{\beta}-\log p}\dfrac{\left(\dfrac{\alpha_{\beta}-\log p}{\alpha_{\beta}}\right)\left(1+\dfrac{\alpha_{\beta}-\log p}{\alpha_{\beta}}\right)}{\left(1-\dfrac{\alpha_{\beta}-\log p}{\alpha_{\beta}}\right)^{3}}=\dfrac{2\alpha_{\beta}^{2}-\alpha_{\beta}\log p}{\left(\log p\right)^2},
    \end{align*}Solving for $\alpha_\beta$, we obtain
    \begin{align*}
        \alpha_{\beta}=\dfrac{\log p+\sqrt{(\log p)^{2}+8(\beta\log p)^{2}}}{4}=\log p\cdot \dfrac{1+\sqrt{1+8\beta^2}}{4},
    \end{align*}
\end{example}

\begin{remark}
    Now we focus on the Haar measure $\mu_{p}$. To compute the associated asymptotic power means $\beta_{q}$ for the support of $\mu_{p}$, one may either proceed by direct integration, as in Proposition \ref{PropCalculosHar}, or apply Theorem \ref{MainTheoremPaper2}. We verify the latter approach. Recall from \cite[Remark 4.2]{paper1} that the Haar measure coincides with $\mu_{\alpha}$ for $\alpha=\frac{p}{p-1}\log p$. For the Harmonic mean, we obtain
    \begin{align*}
        \beta_{-1}^{-1}=\dfrac{\log p}{\alpha-\log p}\log\left(\dfrac{\alpha}{\log p}\right)\bigg\lvert_{\alpha=\frac{p}{p-1}\log p}=(p-1)\log\left(\dfrac{p}{p-1}\right).
    \end{align*}For the quadratic mean,
    \begin{align*}
        \beta_{2}^{2}=2\alpha^{2}-\alpha\log p\bigg\lvert_{\alpha=\frac{p}{p-1}\log p}=\dfrac{p(p+1)}{(p-1)^{2}}.
    \end{align*} In a similar way, one recovers the asymptotic geometric mean from Proposition \ref{derivative in zero}.
\end{remark}

\section*{Acknowledgements}
We thank Godofredo Iommi for comments on the writing and structure of the article.

M.A. was supported by ANID Fondecyt postdoctoral grant 3261344 from Chile.

\bibliographystyle{amsalpha}
\bibliography{refs.bib}

\end{document}